\documentclass[11pt, reqno]{amsart}
\usepackage{amsmath, amsthm, amscd, amsfonts, amssymb, graphicx, color}
\usepackage[bookmarksnumbered, colorlinks, plainpages]{hyperref}

\makeatletter \oddsidemargin.9375in \evensidemargin \oddsidemargin
\newtheorem{theorem}{Theorem}[section]

\newtheorem{proposition}[theorem]{Proposition}
\newtheorem{corollary}[theorem]{Corollary}
\theoremstyle{definition}

\theoremstyle{approach}

\numberwithin{equation}{section}

\begin{document}
	\setcounter{page}{1}
	
	\title{BSE-properties of Vector-valued group algebras}
	
	\author[Mitra Amiri and Ali {Rejali}$^*$]{Mitra Amiri and Ali {Rejali}$^*$}
	
	\thanks{* Corresponding auther}
	
	
\begin{abstract}
Let $ \mathcal{A} $ be a commutative and semisimple Banach algebra with identity norm one and $ G $ be an abelian locally compact Hausdorff group. In this paper, we study BSE-Property for $L^1(G,\mathcal A)$ and show that $L^1(G,\mathcal A)$ is a BSE algebra if and only if $\mathcal A$ is so.
\end{abstract}

\maketitle

\section{Introduction and preliminaries}

Let $G$ be an abelian locally compact Hausdorff group and $\mathcal A$
be a commutative Banach algebra. It is shown in \cite{AmRe}, the conditions that 
$L^p(G,\mathcal A)$ $(1<p<\infty)$ is a BSE-algebra. By \cite[Proposition 1.5.4.]{K},
$L^1 (G,\mathcal A)$ is isometrically isomorphic with the projective tensor product $L^1(G)$ and $\mathcal A$, denoted by $L^1(G)\widehat{\otimes}\mathcal A$.
It follows that $L^1 (G,\mathcal A)$ is always a Banach algebra under convolution product. We verify the BSE property for $L^1(G,\mathcal A)$.

The acronym "BSE" stands for Bochner-Schoenberg-Eberlein and refers to a famous
theorem, for the additive group of real numbers, proved by
Bochner and Schoenberg \cite{Boch,SHo}. Then the result was generalized by Eberlein
\cite{Ebe}, for an abelian locally compact group $G$, which is indicating the
BSE-property of the group algebra $L^1(G)$ \cite{Rud}.
This result, has led Takahashi and Hatori \cite{Tak1} to introduce the BSE-property for any
commutative and without order Banach algebra $\mathcal A$. 

Let $(\mathcal A,\|\cdot\|_{\mathcal A})$ be a without order
commutative Banach algebra, in the sense that $a{\mathcal A}=\{0\}$
implies $a=0$ $(a\in {\mathcal A})$. We denote by
$\Delta({\mathcal A})$, the space consisting of all nonzero multiplicative linear
functionals on $\mathcal A$, called the character space of $\mathcal A$.
Throughout the paper, we assume that $\Delta({\mathcal A})$ is nonempty.
It should be noted that $\Delta(\mathcal A)$, equipped with the weak$^*$topology,
inherited from ${\mathcal A}^*$, is a locally compact Hausdorff space.
We denote by $C_b(\Delta(\mathcal A))$ the space consisting of all bounded and continuous
functions on $\Delta(\mathcal A)$.
A bounded net $\{e_{\alpha}\}_{\alpha\in I}$ in $\mathcal A$, is called a bounded
$\Delta$-weak approximate identity for $\mathcal A$ if
$\lim_\alpha\varphi(ae_\alpha)=\varphi(a),$ for all $a\in\mathcal A$ and $\varphi\in\Delta(\mathcal A)$;
see \cite{JL}. Following \cite{Larsen}, a bounded linear operator
$T$ on $\mathcal A$, satisfying $T(ab)=aT(b)$, $(a,b\in \mathcal A)$,
is called a multiplier. The set of all multipliers on $\mathcal A$ is
denoted by $M({\mathcal A})$, which is a unital commutative Banach
algebra, called the multiplier algebra of ${\mathcal A}$.
By \cite[Theorem 1.2.2]{Larsen}, for any $T\in M({\mathcal A})$,
there exists a unique function $\widehat{T}\in C_b(\Delta(\mathcal A))$
such that $$\widehat{T(a)}(\varphi)=\widehat{T}(\varphi)\widehat{a}(\varphi),$$ for
all $a\in {\mathcal A}$ and $\varphi\in \Delta({\mathcal A})$.
A bounded and continuous function
$\sigma$ on $\Delta({\mathcal A})$ is called a BSE-function if
there exists a constant $\beta>0$ such that following the
inequality holds:
\begin{equation}\label{e15}
\left\vert \displaystyle\sum_{k=1}^n\alpha_k\sigma(\varphi_k)\right\vert\leq \beta\;\left\Vert
\displaystyle\sum_{k=1}^n\alpha_k\varphi_k\right\Vert_{{\mathcal
		A}^*},
\end{equation}
for any finite number of  $\alpha_1,...,\alpha_n\in\Bbb C$ and the same number of
$\varphi_1,...,\varphi_n\in \Delta({\mathcal A})$.
The BSE-norm of $\sigma$ is the infimum of all $\beta$, which satisfying \eqref{e15}
and will be denoted by $\Vert\sigma\Vert_{BSE}$. By \cite{Tak1}, $C_{BSE}(\Delta({\mathcal A}))$,
the space consisting of all BSE-functions on $\Delta(\mathcal A)$, is a commutative and semisimple
Banach algebra, equipped with the norm $\Vert\cdot\Vert_{BSE}$ and pointwise product.
Then ${\mathcal A}$ is called a BSE-algebra (or has the BSE-property)
$\widehat{M({\mathcal A})}=C_{BSE}(\Delta({\mathcal A})),$ where
$\widehat{M({\mathcal A})}=\{\widehat{T}:T\in M({\mathcal
	A})\}.$ By \cite[Corollary 5]{Tak1},
$\widehat{{M}({\mathcal A})}\subseteq  C_{BSE}(\Delta({\mathcal A})),$
if and only if ${\mathcal A}$ has a bounded
$\Delta$-weak approximate identity. It follows that all BSE-algebras possesses a bounded $\Delta$-weak
approximate identity. Moreover, $\mathcal A$ is called a BSE-algebra of type $I$
if
$$
C_b(\Delta({\mathcal A}))=C_{BSE}(\Delta({\mathcal A}))=\widehat{M({\mathcal A})}.
$$
The Gelfand mapping $\Gamma_{\mathcal A}: \mathcal A\rightarrow C_b(\Delta(\mathcal A))$ is defined by
$a\mapsto\widehat{a},$ such that $\widehat{a}$ is the Gelfand
transform of $a$. The Banach algebra  $\mathcal A$
is called semisimple if $ker(\Gamma_{\mathcal A})=\{0\}.$
It should be noted that we always have $\widehat{\mathcal A}\subseteq
C_{BSE}(\Delta({\mathcal A}))$ and for any
$a\in\mathcal A$,
$$\|\widehat{a}\|_\infty\leq\|\widehat{a}\|_{BSE}\leq\|a\|_{\mathcal A},$$
where $\widehat{\mathcal A}$ is the range
of Gelfand mapping of $\mathcal A$. Let
${\mathcal M}({\mathcal A})$ is the normed algebra, consisting of all $\Phi\in C_b(\Delta(\mathcal A))$ such that $\Phi\cdot\widehat{\mathcal A}\subseteq\widehat{\mathcal A}.$
If $\mathcal A$ is semisimple then $\widehat{M({\mathcal
		A})}={\mathcal M}({\mathcal A})$ \cite[p. 30]{Larsen}. Consequently, a semisimple and
commutative Banach algebra $\mathcal A$ is a BSE-algebra if
$C_{BSE}(\Delta({\mathcal A}))=\mathcal{M}({\mathcal A}).$ The interested reader is referred to \cite{AKT},
\cite{Dales}, \cite{IT1}, \cite{I}, \cite{Kan}, \cite{K},
\cite{Tak2}, and \cite{Tak3}. In this paper, we study this property for $L^1(G,\mathcal A)$ and show that $L^1(G,\mathcal A)$ is a BSE algebra,
if and only if $\mathcal A$ is so.

\section{\bf Some basic results}

Let $(\mathcal A,\|\cdot\|_{\mathcal A}) $ be a Banach algebra,  $G$ be a locally compact Hausdorff group with a left Haar measure $\lambda$ and if $1\leq p<\infty $. in the case where
$1<p<\infty$, we assume in addition that $\mathcal A$ is separable. A function $f: G\longrightarrow\mathcal A$ is called strongly measurable, if $f$ is Borel measurable and also $f(G)$ is separable in $\mathcal A$. Thus in the case where $\mathcal A$ is separable, the concept of measurability and strong measurability are equivalent. Most of the properties of integral theory in the complex version, are also valid for the vector-valued case. We refer to \cite{D}, as a complete survey in this issue.

For any Borel measurable function $f: G\rightarrow\mathcal A$, let
$$\lVert f \rVert_{p,\mathcal A}=\left(\int_{G} \lVert f(x) \rVert_{\mathcal A}^{p} d\lambda(x )\right)^{\frac{1}{p}}.$$
Then $L^p (G,\mathcal A)$ is the Banach space, consisting of all Borel measurable functions $f: G \longrightarrow \mathcal A$ such that $\|f \Vert_{p,\mathcal A}<\infty$. For each $ f \in L^p (G,\mathcal A) $, define $ \overline{f}(x)= \Vert f(x) \Vert_{\mathcal A} $, for all $ x \in G $. Then $ f \in L^p (G,\mathcal A) $ if and only if $ \overline{f} \in L^p (G) $.
In this case, $\|\overline{f}\|_{p}=\|f\|_{p,\mathcal A}$. Recall that
for measurable vector-valued functions $f,g: G\longrightarrow\mathcal A$, the convolution multiplication
$$
(f *g)(x) =\int_Gf(y) g(y^{-1}x) d\lambda(y),
$$
is defined at each point $x \in G$ for which this makes sense.

Note that by a famous conjecture, proved by Saeki \cite{S}, $L^p(G)$ is closed under convolution product if and only if $G$ is compact.

Let $\widehat{G}$ be dual group of $G$, introduced in \cite{Rud}. By \cite[Proposition 1.5.4]{K}, $L^1(G,\mathcal A)$ is
isomorphic with the projective tensor product $L^1(G)$ and $\mathcal A$.
Moreover, by \cite[Theorem 2.11.2]{K}, the character space of $L^1(G,\mathcal A)$ is
homeomorphic with $\widehat{G}\times\Delta(\mathcal A)$, such that for all $\chi\in\widehat{G}$ and $\varphi\in\Delta(\mathcal A)$,
$$
\chi\otimes\varphi:L^1(G)\widehat{\otimes}\mathcal A\longrightarrow\Bbb C,
$$
defined as, $$(\chi\otimes\varphi)(f\otimes a)=\chi(f)\varphi(a)\;\;\;\;(f\in L^1(G),a\in\mathcal A).$$
Moreover, for all $f\in L^1(G,\mathcal A)$, we have
$$
(\chi\otimes\varphi)(f)=\varphi\left(\int_G\overline{\chi(x)}f(x)d\lambda(x)\right)=
\int_G\overline{\chi(x)}\varphi(f(x))d\lambda(x).
$$
Now Corollary \cite[Corollary 2.6]{AmRe} together with \cite[Lemma 2.2]{ANR} yield the next result.

\begin{proposition}
	Let $\mathcal A$ be a separable commutative Banach algebra, $G$ be
	an abelian compact Hausdorff group and $1\leq p<\infty$. Then
	$$
	\Delta(L^p(G,\mathcal A))=\{(\chi\otimes\varphi)|_{L^p(G,\mathcal A)}:\;\;\chi\in\widehat{G},\varphi\in\Delta(\mathcal A)\}.
	$$
\end{proposition}

\section{\bf $L^1(G,\mathcal A)$ as a BSE-algebra}

Let $G$ be an abelian locally compact group and $\mathcal A$
be a commutative Banach algebra. In this section, we verify the BSE property for $L^1(G,\mathcal A)$.
By Corollary 2.11.3 and Theorem 2.11.8 of \cite{K}, $L^1(G,\mathcal A)$ is semisimple if and only if $\mathcal A$ is so.
Now we provide a necessary and sufficient for $L^1(G,\mathcal A)$ to be without order. First, we define some
important linear continuous maps as follows. For any $\varphi\in L^\infty(G)$ and $g\in {\mathcal A}^*$,
define $\theta_\varphi: L^1(G,\mathcal A)\rightarrow\mathcal A$ and
$\theta_g: L^1(G,\mathcal A)\rightarrow L^1(G)$
by
$$\theta_\varphi\left(\sum_{n=1}^{\infty}f_n\otimes a_n\right)=\sum_{n=1}^{\infty}\varphi(f_n)a_n
\;\;\text{and}\;\;\theta_g\left(\sum_{n=1}^{\infty}f_n\otimes a_n\right)=\sum_{n=1}^{\infty}f_ng(a_n).$$ It is easily verified that  $\|\theta_\varphi\|\leq\|\varphi\|_\infty$ and $\|\theta_g\|\leq\|g\|_{{\mathcal A}^*}.$

\begin{theorem}
	Let $\mathcal A $ be a commutative and semisimple Banach algebra and
	$G$ be an abelian locally compact Hausdorff group. Then $L^1(G,\mathcal A)$
	is without order if and only if $\mathcal A$ is so.
\end{theorem}

\begin{proof}
	Suppose that $L^1(G,\mathcal A)$ is without order and $a\in\mathcal A$ is nonzero.
	We show that there is $c\in\mathcal A$ such that $ac\neq 0$.
	Take $f\in L^1(G)$ to be nonzero. Thus $f\otimes a\neq 0$ and so by the hypothesis there exists
	$u\in L^1(G,\mathcal A)$ as
	$$
	u=\sum_{n=1}^\infty g_n\otimes b_n,
	$$
	with $\sum_{n=1}^\infty \|g_n\|_1\|b_n\|<\infty$  such that $(f\otimes a)u\neq 0$.  Consequently
	there exists $n_0\in \Bbb N$ such that $f*g_n\otimes ab_{n_0}\neq 0,$ which implies that  $ab_{n_0}\neq 0$.
	
	Conversely, let $u\in L^1(G,\mathcal A)$ such that $uv=0$, for all $v\in L^1(G,\mathcal A)$.
	We prove that $u=0$. Thus $u=\sum_{n=1}^\infty f_n\otimes a_n$ with $\sum_{n=1}^\infty \|f_n\|_1\|a_n\|<\infty$.
	By the hypothesis, for all $f\otimes a\in L^1(G,\mathcal A)$ we have
	$$
	u(f\otimes a)=\sum_{n=1}^\infty f_n*f\otimes a_na=0.
	$$
	It follows that  $\theta_{\varphi}(fu)a=0$, for all $\varphi\in L^\infty(G)$ and
	$a\in \mathcal A$. Since $\mathcal A$ is without order, we obtain
	$$\theta_{\varphi}(fu)=\sum_{n=1}^\infty \varphi(f*f_n)a_n=0,$$
	for all $f\in L^1(G)$. Similarly, by the fact that $L^1(G)$ is without order, we obtain for all $g\in {\mathcal A}^*$
	and $a\in\mathcal A$  that
	$$\theta_{g}(ua)=\sum_{n=1}^\infty f_ng(a_na)=0.$$
	It follows that for all $f\in L^1(G)$, $a\in\mathcal A$, $\varphi\in L^\infty(G)$ and $g\in {\mathcal A}^{*}$
	$$\sum_{n=1}^\infty \varphi(f*f_n)g(a_n)=\sum_{n=1}^\infty \varphi(f_n)g(a_na)=0.$$
	Consequently,  for all $\varphi\in L^\infty(G)$ and $g\in {\mathcal A}^{*}$,
	we obtain
	$$fu(\varphi,g)=ua(\varphi,g)=0,$$
	which implies that $fu=ua=0$.
	Thus for all $g\in {\mathcal A}^{*}$,
	$$0=\theta_{g}(fu)=\sum_{n=1}^{\infty} f*f_n g(a_n)=f\left(\sum_{n=1}^{\infty} f_n g(a_n)\right)$$
	and since $L^1(G)$ is without order, we get $$\sum_{n=1}^{\infty}f_ng(a_n)=0.$$
	Thus for all $\varphi\in L^\infty(G)$ we have
	$$\varphi\left(\sum_{n=1}^{\infty}f_ng(a_n)\right)=0=\sum_{n=1}^{\infty}\varphi(f_n)g(a_n).$$
	It follows that
	$$u(\varphi,g)=\sum_{n=1}^{\infty}\varphi(f_n)g(a_n)=0\;\;\;\;\;\;(\varphi\in L^\infty(G),g\in {\mathcal A}^*),$$
	and so $u=0$, as claimed.
\end{proof}

\begin{proposition}\label{p100}
	Let $\mathcal A $ be a commutative and semisimple Banach algebra and
	$G$ be an abelian locally compact Hausdorff group. Then $L^1(G,\mathcal A)$ has a bounded
	$\Delta-$weak approximate identity if and only if $\mathcal A$ has.
\end{proposition}

\begin{proof}
	Suppose that $\{u_\alpha\}_{\alpha\in\Lambda}$ is a bounded $\Delta-$weak approximate identity for $L^1(G,\mathcal A)$, bounded
	by $M>0$. Thus for all $\chi\in \widehat{G}$ and  $\varphi\in \Delta(\mathcal A)$ we have $\lim_{\alpha}(\chi\otimes\varphi)(u_\alpha)=1$. Take $\chi_0\in\widehat{G}$ and $\varphi_0\in\Delta(\mathcal A)$  to be fixed and define
	$$e_\alpha:=\theta_{\varphi_{0}}(u_{\alpha}) \in L^1(G)\;\;\;\;\;\;(\alpha\in\Lambda)$$
	and
	$$f_\alpha:=\theta_{\chi_{0}}(u_{\alpha}) \in \mathcal A\;\;\;\;\;\;(\alpha\in\Lambda)$$
	Then it is easily verified that $\{e_{\alpha}\}_{\alpha\in\Lambda}$ and $\{f_{\alpha}\}_{\alpha\in\Lambda}$ are bounded $\Delta-$weak approximate identity for $L^1(G)$
	and $\mathcal A$, respectively.
	
	Conversely, suppose that $(e_\alpha)$ and $(f_\beta)$ are bounded $\Delta-$weak approximate identity for $L^1(G)$
	and $\mathcal A$, respectively. It is easily obtained that $(e_\alpha\otimes f_\beta)_{\alpha,\beta}$ is a bounded
	$\Delta-$weak approximate identity for $L^1(G,\mathcal A)$.
\end{proof}

Before we proceed to the main result of this section, we require the following theorem. We recal that $\overline{G}$ is the Bohr compactification of $G$, introduced in \cite{Rud}. 

\begin{theorem}\label{t555}
	Let $\mathcal A$ be a unital BSE Banach algebra and $ G $ be an abelian  locally compact Hausdorff group. Then
	$$ C_{BSE} \big( \Delta ( L^1 (G, \mathcal{A} ))\big)=M_b\big(G,C_{BSE}(\Delta (\mathcal{A}))\big)^{\wedge}.$$
\end{theorem}

\begin{proof}
	By the hypothesis,  $ \mathcal{A}=\widehat{\mathcal{A}}=C_{BSE} ( \Delta ( \mathcal{A} ))$. For convinience, let $B:=M_b \big(G, C_{BSE} ( \Delta ( \mathcal{A} ))\big)$ and 
	take $\mu\in B$. Thus $ \widehat{\mu} \in C_b ( \Delta (B)) $ and $ \widehat{\mu}(\psi)=\psi(\mu)$,  for any $ \psi \in \Delta (B) $. Since $L^1(G,\mathcal{A})$ is an ideal in $ M_b (G,\mathcal A) $, thus
	$$
	\Delta \big( L^1 (G , \mathcal{A})\big)=\lbrace  \psi \vert_{L^1 (G, \mathcal{A})}  : \, \psi \in \Delta (B)\rbrace 
	= \widehat{G} \otimes \Delta ( \mathcal{A}).
	$$
	Suppose that $M$ is the closed linear span of $\Delta(\mathcal{A})$ and $\overline{G}$ is the Bohr compactification of 
	$G$. For any $\chi\in\widehat{G}$ and $\varphi\in\Delta(\mathcal A)$, define
	$$  f_{\chi \otimes \varphi} ( \gamma )=\overline{\gamma ( \chi ) }  \varphi \;\;\;\;\;\;\;\;\;\;(\gamma \in \overline{G})$$
	and for any complex numbers $ \lambda = ( c_i )  $,  $ \psi = ( \psi_i )\subseteq\Delta(L^1(G,\mathcal A))$ we have
	$$  f_{\psi , \lambda} ( \gamma ) = \sum_{i=1}^{n} c_i \overline{\gamma ( \chi_i ) } \varphi_i  \quad ( \gamma \in \overline{G} )$$
	such that $\psi_i=\chi_i \otimes \varphi_i$, for each $1\leq i \leq n$. We prove that $f_{\psi , \lambda} \in 
	C (\overline{G},M)$. In fact,
	$$  \big\Vert f_{\psi , \lambda} \big\Vert_{\infty , M} = \sup \left\{\big\Vert \sum_{i=1}^{n} c_i \overline{\gamma ( \chi_i ) } \varphi_i \big\Vert_{\mathcal{A}^{\star}} : \, \gamma \in \overline{G}\right\}$$
	and
	$$  \big\Vert \sum_{i=1}^{n} c_i \overline{\gamma ( \chi_i ) } \varphi_i \big\Vert_{\mathcal{A}^{\star}} = \underset{\big\Vert a \big\Vert_{\mathcal{A}} \leq 1}{\sup}\left\{\big\vert \sum_{i=1}^{n}  c_i \overline{\gamma ( \chi_i ) } \varphi_i (a) \big\vert\right\}.$$
	Therefore
	\begin{align*}
	\big\Vert f_{\psi , \lambda} \big\Vert_{\infty , M} &= \underset{\gamma \in \overline{G} }{\sup} \,\,\, \underset{\Vert a \Vert \leq 1}{\sup} \big\lbrace  \big\vert \sum_{i=1}^{n}  c_i \overline{\gamma ( \chi_i ) } \varphi_i (a) \big\vert  \big\rbrace \\
	&= \underset{\Vert a \Vert \leq 1}{\sup} \,\,\, \underset{\gamma \in \overline{G} }{\sup}  \big\lbrace  \big\vert \sum_{i=1}^{n}  c_i \overline{\gamma ( \chi_i ) } \varphi_i (a) \big\vert  \big\rbrace \\
	&= \underset{\Vert a \Vert \leq 1}{\sup} \,\,\, \underset{x \in G }{\sup}  \big\lbrace  \big\vert \sum_{i=1}^{n}  c_i \overline{\chi_i ( x ) } \varphi_i (a) \big\vert  \big\rbrace \\
	&= \underset{\Vert a \Vert \leq 1}{\sup} \,\,\, \underset{x \in G }{\sup} \big\lbrace \big\vert \delta_{x^{-1}} \otimes \widehat{a} \big(\sum_{i=1}^{n}  c_i  \chi_i \otimes \varphi_i \big) \big\vert  \big\rbrace  \\
	&\leq \underset{\Vert a \Vert \leq 1}{\sup} \,\,\, \underset{x \in G }{\sup} \big\Vert \delta_{x^{-1}} \otimes \widehat{a} \big\Vert \cdot \big\Vert  \sum_{i=1}^{n}  c_i  ( \chi_i \otimes \varphi_i )  \big\Vert \\
	&\leq \big\Vert  \sum_{i=1}^{n}  c_i ( \chi_i \otimes \varphi_i )  \big\Vert_{L^1(G, \mathcal{A} )^{\star}}  .
	\end{align*}
	This implies that $ f_{\psi , \lambda} $ is bounded. Moreover, 
	if $ \gamma_{\alpha}  \overset{w^{\star}}{\longrightarrow} \gamma$ in $ \overline{G} $, then
	$$
	f_{\psi , \lambda} ( \gamma_{\alpha}  ) =\sum_{i=1}^{n}  c_i \overline{\gamma_{\alpha} ( \chi_i )}  \varphi_i 
	\longrightarrow  \sum_{i=1}^{n}  c_i \overline{\gamma ( \chi_i )}  \varphi_i = f_{\psi , \lambda} ( \gamma ) .
	$$
	Therefore $ f_{\psi , \lambda} $ is continuous and so $ f_{\psi , \lambda} \in C ( \overline{G} , M ) $.
	Now suppose that
	$T: \langle f_{\chi \otimes \varphi}:\chi\in\widehat{G},\varphi\in\Delta(\mathcal A) \rangle\longrightarrow \mathbb{C}$, defined as
	$$T(f_{\chi \otimes \varphi})=\int_{G} \overline{\chi (x)}\mu (x)(\varphi) dx.$$
	Let $\overline{T}$ belonging to $C( \overline{G}, M)^{*}=M_b (\overline{G}, M^*) $ is the extension of the functional 
	$T$. Then $\widehat{\mu}(\chi\otimes \varphi )=\overline{\mu}(f_{\chi \otimes \varphi})$, where $ \overline{\mu} $ is the corresponding measure of the functional $\overline{T}$.
	For any complex numbers $c_1,\cdots,c_n$ and the same number of $\varphi_1,\cdots,\varphi_n \in \Delta ( \mathcal{A}) $ and $\chi_1,\cdots,\chi_n \in \widehat{G}$, we have
	\begin{align*}
	\big\vert \sum_{i=1}^{n}  c_i \widehat{\mu} ( \chi_i \otimes \varphi_i ) \big\vert &=  \big\vert \sum_{i=1}^{n}  c_i \int_{G} \overline{\chi_i (x)} \mu (x) ( \varphi_i ) dx \big\vert\\
	&\leq \int_{G} \big\vert  \sum_{i=1}^{n}  c_i \overline{\chi_i (x)} \mu (x) ( \varphi_i )  \big\vert dx \\
	&= \int_{G} \big\vert  \sum_{i=1}^{n}  c_i ( \delta_{x^{-1}} \otimes \mu (x) ) ( \chi_i \otimes  \varphi_i )  \big\vert dx\\
	&\leq \int_{G} \big\Vert \delta_{x^{-1}} \otimes \mu (x) \big\Vert\;\big\Vert \sum_{i=1}^{n} c_i (\chi_i \otimes  \varphi_i )  \big\Vert dx  \\
	&= \int_{G} \big\Vert \mu (x) \big\Vert_{BSE}\, dx \;\big\Vert  \sum_{i=1}^{n} c_i (\chi_i \otimes  \varphi_i )  \big\Vert_{L^1 (G , \mathcal{A} )^{\star}}\\
	&= \big\Vert \mu \big\Vert_{1 , BSE}\;\big\Vert \sum_{i=1}^{n} c_i (\chi_i \otimes  \varphi_i )  \big\Vert_{L^1 (G , \mathcal{A} )^{\star}},
	\end{align*}
	which implies that $\big\Vert \widehat{\mu} \big\Vert_{BSE} \leq \big\Vert \mu \big\Vert_{1, BSE} $. Moreover, the Gelfand mapping $ \mu \mapsto \widehat{\mu} $ is continuous,  $ \widehat{\mu} \in C_b ( \Delta (B))  $ and $ \big\Vert \widehat{\mu} \big\Vert_{\infty}  \leq \big\Vert \mu \big\Vert$. If $ \psi_{\alpha}  \overset{w^{\star}}{\longrightarrow}\psi $ in 
	$ \Delta \big( L^1 (G, \mathcal{A})\big) $, then $\overline{\psi_{\alpha}} \overset{w^{\star}}{\longrightarrow} \overline{\psi}$ in $ \Delta (B) $, where any $\overline{\psi_{\alpha}}$ is an extension of $\psi_{\alpha}$ such that 
	$\overline{\psi_{\alpha}} \big\vert_{L^1 (G, \mathcal{A})}=\psi_{\alpha}.$ Thus
	\begin{center}
		$ \widehat{\mu} ( \psi_{\alpha} ) = \widehat{\overline{\mu}} ( \overline{\psi_{\alpha}} ) \longrightarrow \widehat{\overline{\mu}} ( \overline{\psi} ) = \widehat{\mu} ( \psi)$
	\end{center}
	and so $ \widehat{\mu} \in C_b ( \Delta ( \mathcal{A})) .$
	This implies that $ \widehat{\mu} \in C_{BSE} \big( \Delta ( \mathcal{A})\big) $. Consequently,
	\begin{equation}\label{ee204}
	\widehat{M_b \big(G , C_{BSE} ( \Delta ( \mathcal{A}))\big)} \subseteq C_{BSE} \big( \Delta ( L^1 (G, \mathcal{A}))\big).
	\end{equation}
	Now, we prove that the reverse of the inclusion \eqref{ee204}. Take $ \sigma \in C_{BSE} \big( \Delta ( L^1 ( G, \mathcal{A}))\big) $. Then for any complex numbers $c_1,\cdots,c_n$ and the same number of $\varphi_1,\cdots,\varphi_n$ belonging to $\Delta ( \mathcal{A}) $ and $\chi_1,\cdots,\chi_n \in \widehat{G}$, we have
	\begin{center}
		$\big\vert \sum_{i=1}^{n} c_i \sigma ( \chi_i \otimes \varphi_i ) \big\vert  \leq \big\Vert \sigma \big\Vert_{BSE}\; \big\Vert \sum_{i=1}^{n} c_i ( \chi_i \otimes \varphi_i ) \big\Vert_{L^1 (G, \mathcal{A} )^{\star}}.$
	\end{center}
	Define $U: \langle f_{\psi , \lambda} \rangle\longrightarrow \mathbb{C}$ by $ U( f_{\psi , \lambda }) := \sum_{i=1}^{n} c_i \sigma ( \chi_i \otimes \varphi_i )$, in which $ \lambda = ( c_i )  $, $ \psi = ( \psi_i ) $ and
	$$  f_{\psi , \lambda} ( \gamma ) = \sum_{i=1}^{n} c_i \overline{\gamma ( \chi_i ) } \varphi_i  \quad ( \gamma \in \overline{G} ).$$
	According to the Hahn–Banach theorem, the functional $ U $ has the extension $ \overline{U} $ over $ C( \overline{G} , M) $ such that
	\begin{center}
		$\overline{U} ( f_{\psi , \lambda} ) = U ( f_{\psi , \lambda} ) = \sum_{i=1}^{n}  c_i \sigma ( \chi_i \otimes \varphi_i ) .$	
	\end{center}
	Suppose that $ \vartheta \in M_b ( \overline{G} , M^{\star} ) $ such that
	\begin{align*}
	U( f_{\psi , \lambda} ) &= \int_{\overline{G}} f_{\psi , \lambda} ( \gamma ) d \vartheta ( \gamma )  \\
	&= \int_G f_{\psi , \lambda} ( \widehat{x}) d \vartheta ( \widehat{x} ) \\
	&:= \int_G \sum_{i=1}^{n} c_i \overline{\chi_i (x)} \varphi_i  d \mu (x) \\
	&= \int_G \mu (x) (  \sum_{i=1}^{n} c_i \overline{\chi_i (x)} \varphi_i  ) dx \\
	&= \int_G   \sum_{i=1}^{n} c_i \overline{\chi_i (x)} \mu (x) ( \varphi_i  ) dx\\
	&:= \widehat{\mu} (  \sum_{i=1}^{n} c_i ( \chi_i \otimes \varphi_i )  )  .
	\end{align*}
	Thus
	$$  \sum_{i=1}^{n} c_i \sigma ( \chi_i \otimes \varphi_i ) =  \widehat{\mu} \big(  \sum_{i=1}^{n} c_i ( \chi_i \otimes \varphi_i )  \big) .$$
	In a special case, $ \sigma ( \chi \otimes \varphi ) = \widehat{\mu} ( \chi \otimes \varphi ) $ and so $ \sigma = \widehat{\mu} $, for some $ \mu \in M_b (G , M^{\star} ) $. Now, we show that the mapping $ x \longmapsto \mu (x) $ belongs to $ M_b \big(G , C_{BSE} ( \Delta ( \mathcal{A} ))\big) $.	For any complex numbers $c_1,\cdots,c_n$
	and the same number of $\varphi_1,\cdots,\varphi_n \in \Delta ( \mathcal{A}) $ we have
	\begin{center}
		$\big\vert \sum_{i=1}^{n}  c_i  \mu (x) ( \varphi_i ) \big\vert = \big\vert \mu (x)  \big(\sum_{i=1}^{n}  c_i  \varphi_i \big) \big\vert
		\leq \big\Vert \mu (x) \big\Vert \cdot \big\Vert \sum_{i=1}^{n}  c_i  \varphi_i \big\Vert_{\mathcal{A}^{\star}} .$
	\end{center}
	Therefore $ \mu (x) \in C_{BSE} ( \Delta ( \mathcal{A})) $, for all $ x \in G $. Then $ \mu \in M_b \big(G, C_{BSE} ( \Delta ( \mathcal{A}))\big) $ such that $ \sigma = \widehat{\mu} $. Consequently,
	\begin{equation}\label{ee205}
	C_{BSE} \big( \Delta (L^1 (G , \mathcal{A}))\big) \subseteq M_b \big(G , C_{BSE} ( \Delta ( \mathcal{A}))\widehat{\big)} .
	\end{equation}
	Now \eqref{ee204} and \eqref{ee205} imply that
	$$  C_{BSE} ( \Delta ( L^1 (G , \mathcal{A} ))) = M_b (G, C_{BSE} ( \Delta ( \mathcal{A}))\widehat{)} . $$
\end{proof}

\begin{theorem}\label{t100}
	Let $ \mathcal{A} $ be a unital commutative and semisimple Banach algebra, with the identity with norm one, and $G$ be an abelian locally compact Hausdorff group. Then $ L^1(G,\mathcal{A})$  is a BSE-algebra if and only if $\mathcal{A}$ is a BSE-algebra.
\end{theorem}

\begin{proof}
	Suppose that $ \mathcal{A} $ is a BSE-algebra. Then by Theorem \ref{p100}, $L^1 (G,\mathcal{A})$ has a bounded $\Delta-$weak approximate identity. By \cite{UB} we have 
	$$M\big(L^1 (G,\mathcal{A} )\big)=M_b (G,\mathcal{A})$$ and so 
	\begin{center}	
		$\widehat{M\big(L^1 (G , \mathcal{A}))}=\widehat{M_b (G , \mathcal{A})}\subseteq C_{BSE} \big( \Delta ( L^1 (G , \mathcal{A} ))\big).$
	\end{center}	
	Moreover, by Theorem \ref{t555}, we have
	$$ C_{BSE} \big( \Delta ( L^1 (G , \mathcal{A} ))\big) \subseteq\widehat{M_b (G, \mathcal{A})}.$$
	As a result, the equality is established and so $L^1 (G,\mathcal{A})$ is a BSE-algebra.
	
	Conversely, let $ L^1 (G , \mathcal{A}) $  be a BSE-algebra.
	By the hypothesis, $L^1(G,\mathcal A)$ has  a bounded $\Delta$-weak approximate identity
	and so by Proposition \ref{p100}, $\mathcal A$ has so. It follows from \cite[Corollary 5]{Tak1} that
	$$\mathcal M(\mathcal A)\subseteq C_{BSE}(\Delta (\mathcal A)).$$
	Now, we prove the reverse of the inclusion. To this end, take $\sigma \in C_{BSE}(\Delta (\mathcal A))$ and $a \in \mathcal A.$ We find $b\in \mathcal A$ such that $\sigma \widehat{a}=\widehat{b}.$
	Take $\chi_{0}\in\widehat{G}$ to be fixed. Then there exists $f_{0}\in L^1(G)$ such that
	$\chi_{0}(f_{0})=1$. Since $\sigma \in C_{BSE}(\Delta (\mathcal A))$ there exist a bounded net $\{x_{\lambda}\}\subseteq \mathcal A$ such that for all $\varphi\in \Delta (\mathcal A)$, we have $\lim_{\lambda}\widehat {x_{\lambda}}(\varphi)=\sigma(\varphi)$.
	It follows that  $$\lim_{\lambda}\widehat{f_0\otimes x_{\lambda}}(\chi\otimes \varphi)=\lim_{\lambda}\chi(f_{0})\varphi(x_{\lambda})=\chi(f_{0})\sigma(\varphi).$$
	Define $$\sigma_{1}: \widehat{G}\times\Delta(\mathcal A)\rightarrow \mathbb C$$
	with $\sigma_{1}(\chi \otimes \varphi)=\sigma(\varphi)\chi(f_{0}).$
	It is easily verified that $\sigma_{1}\in C_{BSE}(\widehat{G}\times\Delta(\mathcal A))$
	and $\|\sigma_1\|_{BSE}\leq\|\sigma\|_{BSE}$.
	Thus there exists $g_0\in L^1(G,\mathcal A)$  such that
	$$
	\sigma_1\widehat{f_0\otimes a}=\widehat{g_0}.
	$$
	Note that $g_0=\sum_{n=1}^\infty g_n\otimes a_n$
	with
	$$
	\sum_{n=1}^\infty\|g_n\|_1\|a_n\|<\infty.
	$$
	Thus for all $\varphi\in\Delta(\mathcal A)$, we have
	$$
	\sigma(\varphi)\varphi(a)=\sigma_1(\chi_0\otimes\varphi)\varphi(a)=\sum_{n=1}^\infty\chi_0(g_n)\varphi(a_n)=
	\widehat{\sum_{n=1}^\infty\chi_0(g_n)a_n}(\varphi).
	$$
	Thus by choosing $b:=\sum_{n=1}^\infty\chi_0(g_n)a_n$, we obtain
	$\sigma\widehat{a}=\widehat{b}$. Therefore the proof is completed.
\end{proof}

\section{\bf $\ell^{1}(G,\mathcal A)$ as a BSE-algebra}
\begin{theorem}\label{t55}
Let $\mathcal A$ be a commutative and semisimple Banach algebra and $G$ be an abelian discrete group. Then 
$$  C_{BSE}\big( \Delta ( \ell^{1} (G , \mathcal{A} ))\big) = \ell^{1} \big( G , C_{BSE} ( \Delta ( \mathcal{A} ))\widehat{\big)}.  $$
 \end{theorem}
\begin{proof}
Suppose that $ \sigma \in C_{BSE}  ( \Delta ( \ell^{1} (G , \mathcal{A} )))$, then for each $ \varphi_i \in \Delta ( \mathcal{A} ) $, $ \chi_i \in \widehat{G} $ and $ c_i \in \mathbb{C} $, put $ \psi  = ( \psi_i )$, $ \psi_i = \chi_i \otimes \varphi_i $ and $ \lambda := ( c_i ) $. Let $ H = ( \widehat{G} , \tau  ) $ and $ \overline{G} = \widehat{H}   $ be the Bohr compactification of $ G $. Define $ f_{\psi , \lambda}  : \overline{G} \longrightarrow M$ in which $ M = \overline{< \Delta ( \mathcal{A} )>}^{\Vert \cdot \Vert}  $ and 
$$  f_{\psi , \lambda}  ( \gamma ) : = \sum_{i=1}^{n}  c_i \gamma ( \chi_i ) \varphi_i .$$
Thus $  f_{\psi , \lambda} \in C ( \overline{G} , M )$. Now if $ T : < f_{\psi , \lambda} > \longrightarrow \mathbb{C} $ define by 
$$T( f_{\psi , \lambda} ) := \sum_{i=1}^{n}  c_i \sigma ( \chi_i \otimes \varphi_i ) , $$
then we have
\begin{equation}\label{ee200}
  \vert T ( f_{\psi , \lambda} )  \vert \leq \Vert \sigma \Vert_{BSE} \cdot \Vert \sum_{i=1}^{n}  c_i  ( \chi_i \otimes \varphi_i ) \Vert_{\ell^{1} ( G , \mathcal{A} )^{\star}}.
 \end{equation} 
We show that the linear mapping $ T $ is well-defined.
Let $  f_{\psi , \lambda} = 0 $. Then 
\begin{align*}
 0 &= \sum_{i=1}^{n} c_i \widehat{x} ( \chi_i ) \varphi_i (a) \quad ( x \in G )\\
 &=\sum_{i=1}^{n} c_i \chi_i ( x ) \varphi_i (a) \quad ( a \in \mathcal{A} ). 
\end{align*}
Therefore
\begin{align*}
\big\Vert \sum_{i=1}^{n} c_i ( \chi_i \otimes \varphi_i ) \big \Vert_{\ell^{1} ( G , \mathcal{A} )^{\star}} &=\sup \big\lbrace \big\vert \sum_{i=1}^{n} c_i \chi_i ( \varphi_i \circ f ) \big\vert  : \, \Vert f \Vert_{1 , \mathcal{A}} \leq 1 \big\rbrace\\
 &=  \sup \big\lbrace \big\vert \sum_{i=1}^{n}  \sum_{x \in G} c_i  \overline{\chi_i (x)} \varphi_i ( f(x) ) \big\vert : \,  \Vert f \Vert_{1 , \mathcal{A}} \leq 1 \big\rbrace \\
 &=0 .
\end{align*}
Now \eqref{ee200} implies that $ T( f_{\psi , \lambda} ) =0 $. By applying Hahn-Banach theorem, there exists $ \overline{T} \in C ( \overline{G} , M )^{\star}  = M_b ( \overline{G} , M^{\star} )$ such that $ \overline{T} $ be an extension of $ T $ and $ \Vert \overline{T} \Vert = \Vert T \Vert $. Suppose that $ \overline{\mu} \in M_b ( \overline{G} , M^{\star} ) $ such that
$$ \sum_{i=1}^{n} c_i  \sigma ( \chi_i \otimes \varphi_i ) = T ( f_{\psi , \lambda} ) = \overline{T} ( f _{\psi , \lambda})  = \overline{\mu} ( f_{\psi , \lambda}  ) = \int_{\overline{G}} f_{\psi , \lambda} ( \gamma ) d \overline{\mu} ( \gamma ) .$$
As a special case
\begin{align*}
\sigma ( \chi \otimes \varphi ) &= \int_{\overline{G}} f_{\chi \otimes \varphi} ( \gamma ) d \overline{\mu} ( \gamma )\\
&= \int_{G} f_{\chi \otimes \varphi} ( \widehat{x}) d \overline{\mu} (\widehat{x}  ) \\
&:=  \int_{G} \chi ( x ) \varphi d \mu (x)\\
&= \sum_{x \in G } \mu (x) ( \chi (x) \varphi )\\
&=  \sum_{x \in G } \chi (x) \mu (x) ( \varphi )\\
&= \int_{G} \overline{\chi (x)} \mu ( x^{-1} ) ( \varphi ) dx \\
&:=\widehat{\nu} ( \chi \otimes \varphi ) ,
\end{align*}
in which $\nu(x)=\mu ( x^{-1} )$. 
Therefore $ \sigma = \widehat{\nu} $. 
Now we show that the mapping $ x \longmapsto \mu ( x^{-1} )  $ is belongs to $ C_{BSE} ( \Delta ( \mathcal{A} ))$ and so $ \mu \in \ell^{1} \big( G , C_{BSE} ( \Delta ( \mathcal{A} ))\big) $. 
Since $ \overline{\mu}  \in M_b ( \overline{G} , M^{\star} ) $, $ \mu \in \ell^{1} ( G , M^{\star} ) $. Therefore
 $$ \Vert \mu \Vert_{1}  = \sum_{x \in G} \Vert \mu (x) \Vert < \infty.$$
Now consider $ \varphi_i \in \Delta ( \mathcal{A} ) $, $ c_i \in \mathbb{C} $ and $ x \in G $. So
\begin{align*}
\big\vert \sum_{i=1}^{n} c_i \mu ( x^{-1}) ( \varphi_i ) \big\vert &= \big\vert \mu ( x^{-1} ) (\sum_{i=1}^{n} c_i \varphi_i ) \big\vert \\
&\leq \Vert \mu ( x^{-1}) \Vert \cdot \Vert \sum_{i=1}^{n} c_i \varphi_i \Vert \\
&\leq \Vert \mu \Vert_1 \cdot \Vert \sum_{i=1}^{n} c_i \varphi_i \Vert.
\end{align*}
Therefore $ \mu \in \ell^{1} \big(G , C_{BSE} ( \Delta ( \mathcal{A}))\big) $ such that $ \sigma = \widehat{\mu} $. Consequently,
\begin{equation}\label{ee201}
 C_{BSE} \big( \Delta ( \ell^{1} ( G , \mathcal{A} ))\big) \subseteq \ell^{1} \big( G , C_{BSE} ( \Delta ( \mathcal{A} ))\big).  
\end{equation}
Conversely, let $ \mu \in \ell^{1} \big( G , C_{BSE} ( \Delta ( \mathcal{A} ))\big) $. For all $ \varphi_i  \in \Delta ( \mathcal{A} )$, $ \chi_i \in \widehat{G} $ and $ c_i \in \mathbb{C} $ we have
\begin{align*}
\big\vert \sum_{i=1}^{n}  c_i \widehat{\mu} ( \chi_i \otimes \varphi_i ) \big\vert &=  \big\vert \sum_{i=1}^{n}  c_i  \sum_{x \in G} \overline{\chi_i (x)} \mu (x) ( \varphi_i ) \big\vert\\
&= \big\vert  \sum_{x \in G} \big( \sum_{i=1}^{n}  c_i \overline{\chi_i (x)} \mu (x) ( \varphi_i ) \big) \big\vert \\
&\leq  \sum_{x \in G} \big\vert  \sum_{i=1}^{n}  \big(c_i \overline{\chi_i (x)}\big) \mu (x) ( \varphi_i )  \big\vert\\
&\leq \sum_{x \in G} \big\vert   ( \sum_{i=1}^{n}  c_i \big( \delta_{x^{-1}} \otimes \mu (x)\big) ( \chi_i \otimes \varphi_i ) \big\vert \\
&= \sum_{x \in G} \big\vert   \big( \delta_{x^{-1}} \otimes \mu (x)\big) \big( \sum_{i=1}^{n}  c_i  ( \chi_i \otimes \varphi_i ) \big)  \big\vert\\
&\leq  \sum_{x \in G} \Vert    \delta_{x^{-1}} \otimes \mu (x) \big\Vert \cdot \Vert   \sum_{i=1}^{n}  c_i  ( \chi_i \otimes \varphi_i ) \big\Vert_{\ell^{1} ( G , \mathcal{A} )^{\star}}\\
&= \sum_{x \in G} \Vert    \delta_{x^{-1}} \Vert \cdot \Vert\mu (x) \Vert_{BSE} \cdot \big\Vert \sum_{i=1}^{n}  c_i  ( \chi_i \otimes \varphi_i ) \big\Vert \\
&=  \sum_{x \in G} \Vert\mu (x) \Vert_{BSE} \cdot \big\Vert \sum_{i=1}^{n}  c_i  ( \chi_i \otimes \varphi_i ) \big\Vert \\
&\leq \big( \sum_{x \in G} \Vert\mu (x) \Vert \big) \big( \big\Vert \sum_{i=1}^{n}  c_i  ( \chi_i \otimes \varphi_i ) \big\Vert  \big)\\
&=\Vert \mu \Vert_1 \cdot \big\Vert \sum_{i=1}^{n}  c_i  ( \chi_i \otimes \varphi_i ) \big\Vert  .
\end{align*}
Thus $ \widehat{\mu} \in C_{BSE} \big( \Delta ( \ell^{1} ( G , \mathcal{A} ))\big).$ This implies that 
\begin{equation}\label{ee202}
\ell^{1} \big( G , C_{BSE} ( \Delta ( \mathcal{A} ))\widehat{\big)}   \subseteq C_{BSE} \big( \Delta ( \ell^{1} ( G , \mathcal{A} ))\big).
\end{equation}
By \eqref{ee201} and \eqref{ee202} we obtain
$$ C_{BSE} ( \Delta ( \ell^{1} ( G, \mathcal{A} ))) = \ell^{1} \big( G, C_{BSE} ( \Delta ( \mathcal{A} ))\big).  $$
\end{proof}

\begin{corollary}
Let	$ \mathcal{A} $ be an unital commutative and semisimple Banach algebra
 and $ G $ be an abelian discrete group. Then the following statements are equivalent.
\begin{enumerate}
\item[(i)] $\ell^{1} (G , \mathcal{A} ) $ is a BSE- algebra.
\item[(ii)] $ \mathcal{A} $ is a BSE-algebra.
\end{enumerate}
\end{corollary}

\begin{proof}
$ (i) \Rightarrow (ii) $ 
Since $ \ell^{1} (G, \mathcal{A} ) = \ell^{1} (G) \widehat{\otimes} \mathcal{A} $, $ \delta_e \otimes 1_{\mathcal{A}} $ is identity element of $ \ell^{1} ( G , \mathcal{A} ) $. Moreover, $ \ell^{1} ( G , \mathcal{A} ) $ is a commutative Banach algebra. Therefore
$$ M \big( \ell^{1} (G , \mathcal{A} )\big) = \lbrace  L_f  :\, f \in \ell^{1} (G , \mathcal{A})  \rbrace  $$
in which $ L_f (g) = f \star g $, for each $ g \in \ell^{1} (G , \mathcal{A} ) $. Also $ \ell^{1} (G , \mathcal{A}) $ is semisimple. Thus
\begin{align*}
M \big(\ell^{1} (G , \mathcal{A} )\widehat{\big)} &= \lbrace  \widehat{L_f} : \, f \in \ell^{1} ( G , \mathcal{A})  \rbrace \\
&= \lbrace  \widehat{f} : \, f \in \ell^{1} (G , \mathcal{A} )  \rbrace \\
&= \ell^{1} (G , \mathcal{A}\widehat{)}\\
&=C_{BSE} \big( \Delta ( \ell^{1} (G , \mathcal{A} ))\big).
\end{align*}
Consequently,
$$ \ell^{1} (G , \mathcal{A} \widehat{)} = \ell^{1} \big(G , C_{BSE} ( \Delta ( \mathcal{A} )) \widehat{\big)}. $$
Now, we prove that $ C_{BSE} ( \Delta ( \mathcal{A} ))  $ is semisimple. By the assumption, $ \mathcal{A} $ is semisimple. Take  $ \sigma_1 , \sigma_2 \in C_{BSE} ( \Delta ( \mathcal{A} )) $ with $ \sigma_1 \neq \sigma_2 $. Then there exists $ \varphi \in \Delta (\mathcal{A}) $ such that $ \sigma_1 ( \varphi ) \neq \sigma_2 ( \varphi )$. Therefore $ \widehat{\varphi} ( \sigma_1 ) \neq  \widehat{\varphi} ( \sigma_2 )$. But
$$ \widehat{\varphi} ( \xi_1 \cdot \xi_2 ) = \xi_1 \cdot \xi_2 ( \varphi ) = \xi_1 (\varphi ) \xi_2 (\varphi ) = \widehat{\varphi} (\xi_1 ) \widehat{\varphi} ( \xi_2 ),  $$
for each $ \xi_1 , \xi_2 \in C_{BSE} ( \Delta ( \mathcal{A} )) $. Hence $ \widehat{\varphi} $ is multiplication and so $ \widehat{\varphi} \in \Delta \big( C_{BSE} ( \Delta ( \mathcal{A} ))\big) $ and separate the points  $ \sigma_1 $ and $ \sigma_2 $. It follows that
 $ \ell^{1} \big( G , C_{BSE} ( \Delta ( \mathcal{A} ))\big) $ is semisimple. Thus the Gelfand mapping is injective and 
$$  \ell^{1} \big(G , \widehat{\mathcal{A}} \big) = \ell^{1} \big(G , C_{BSE} ( \Delta ( \mathcal{A} ))\big) .$$
By \cite[Theorem 4.3]{PH} , $ \widehat{\mathcal{A}}  = C_{BSE} \big( \Delta ( \mathcal{A} )\big)$ and so $ \mathcal{A} $ is a BSE-algebra.

$ (ii) \Rightarrow  (i) $
 Suppose that $ \mathcal{A} $ is a BSE-algebra. Then $ C_{BSE} \big( \Delta ( \mathcal{A} )\big) = \widehat{\mathcal{A}} $. Moreover, by the semisimplicity of $ \mathcal{A} $, $ \mathcal{A} = \widehat{\mathcal{A}} $ as two Banach algebra. Thus
\begin{align*}
C_{BSE} \big( \Delta ( \ell^{1} (G , \mathcal{A} ))\big) &= \ell^{1} \big( G , C_{BSE} ( \Delta ( \mathcal{A} )) \widehat{\big)} \\
&= \ell^{1} (G , \widehat{\mathcal{A}} \widehat{)}\\
&= \ell^{1} (G , \mathcal{A}\widehat{)}.
\end{align*}
Hence $ \ell^{1} (G , \mathcal{A}) $ is a BSE-algebra.
\end{proof}

\vspace{9mm}

{\footnotesize \noindent
	M. Amiri\\
	Department of Pure Mathematis\\
	Faculty of Mathematics and Statistics\\
	University of Isfahan\\
	Isfahan, 81746-73441\\
	Iran\\
	mitra75amiri@gmail.com\\
	m.amiri@sci.ui.ac.ir\\
	
	\noindent
	A. Rejali\\
	Department of Pure Mathematis\\
	Faculty of Mathematics and Statistics\\
	University of Isfahan\\
	Isfahan, 81746-73441\\
	Iran\\
	rejali@sci.ui.ac.ir\\
	a.rejali20201399@gmail.com

\end{document}